\documentclass[12pt]{amsart}
\usepackage{hyperref} 
\usepackage{amsmath, amsthm, amssymb}
\usepackage{hyperref} 
\usepackage{enumerate}
\usepackage{verbatim}
\usepackage{esint}
\usepackage[T1]{fontenc}

\usepackage{tikz,amsthm,amsmath,amstext,amssymb,amscd,epsfig,euscript, mathrsfs,
 dsfont,pspicture,
multicol,graphpap,graphics,graphicx,times,enumerate,subfig,
sidecap,
wrapfig,color,pict2e}
\usepackage{setspace}

\numberwithin{equation}{section}

\DeclareMathOperator{\card}{card}

\title{Bi-Lipschitz embeddings of quasiconformal trees}
\date{\today}

\author{Guy C. David}
\address{Department of Mathematical Sciences\\ Ball State University\\ Muncie, IN 47306}
\email{gcdavid@bsu.edu}

\author{Sylvester Eriksson-Bique}
\address{Research Unit of Mathematical Sciences,
P.O.Box 3000,
FI-90014 Oulu, Finland}
\email{sylvester.eriksson-bique@oulu.fi}

\author{Vyron Vellis}
\address{Department of Mathematics\\ The University of Tennessee\\ Knoxville, TN 37966}
\email{vvellis@utk.edu}

\thanks{G.~C.~David was partially supported by NSF DMS grant 1758709. S.~Eriksson-Bique was partially supported by the Finnish Academy grant  345005. V.~Vellis was partially supported by NSF DMS grant 1952510.}

\begin{document}
%\maketitle

%Theorems
\theoremstyle{plain}
\newtheorem{theorem}{Theorem}
\newtheorem{exercise}{Exercise}
\newtheorem{corollary}[theorem]{Corollary}
\newtheorem{scholium}[theorem]{Scholium}
\newtheorem{claim}[theorem]{Claim}
\newtheorem{lemma}[theorem]{Lemma}
\newtheorem{sublemma}[theorem]{Lemma}
\newtheorem{proposition}[theorem]{Proposition}
\newtheorem{conjecture}[theorem]{Conjecture}
\newtheorem{maintheorem}{Theorem}
\newtheorem{maincor}[maintheorem]{Corollary}
\newtheorem{mainproposition}[maintheorem]{Proposition}
\renewcommand{\themaintheorem}{\Alph{maintheorem}}

\theoremstyle{definition}
\newtheorem{fact}[theorem]{Fact}
\newtheorem{example}[theorem]{Example}
\newtheorem{definition}[theorem]{Definition}
\newtheorem{remark}[theorem]{Remark}
\newtheorem{question}[theorem]{Question}

\numberwithin{equation}{section}
\numberwithin{theorem}{section}

\newcommand{\cG}{\mathcal{G}}
\newcommand{\RR}{\mathbb{R}}
\newcommand{\HH}{\mathcal{H}}
\newcommand{\LIP}{\textnormal{LIP}}
\newcommand{\Lip}{\textnormal{Lip}}
\newcommand{\Tan}{\textnormal{Tan}}
\newcommand{\length}{\textnormal{length}}
\newcommand{\dist}{\textnormal{dist}}
\newcommand{\diam}{\textnormal{diam}}
\newcommand{\vol}{\textnormal{vol}}
\newcommand{\rad}{\textnormal{rad}}
\newcommand{\side}{\textnormal{side}}

%Alphabets
\def\bA{{\mathbb{A}}}
\def\bB{{\mathbb{B}}}
\def\bC{{\mathbb{C}}}
\def\bD{{\mathbb{D}}}
\def\bR{{\mathbb{R}}}
\def\bS{{\mathbb{S}}}
\def\bO{{\mathbb{O}}}
\def\bE{{\mathbb{E}}}
\def\bF{{\mathbb{F}}}
\def\bH{{\mathbb{H}}}
\def\bI{{\mathbb{I}}}
\def\bT{{\mathbb{T}}}
\def\bZ{{\mathbb{Z}}}
\def\bX{{\mathbb{X}}}
\def\bP{{\mathbb{P}}}
\def\bN{{\mathbb{N}}}
\def\bQ{{\mathbb{Q}}}
\def\bK{{\mathbb{K}}}
\def\bG{{\mathbb{G}}}

\def\nrj{{\mathcal{E}}}
\def\cA{{\mathscr{A}}}
\def\cB{{\mathscr{B}}}
\def\cC{{\mathscr{C}}}
\def\cD{{\mathscr{D}}}
\def\cE{{\mathscr{E}}}
\def\cF{{\mathscr{F}}}
\def\cB{{\mathscr{G}}}
\def\cH{{\mathscr{H}}}
\def\cI{{\mathscr{I}}}
\def\cJ{{\mathscr{J}}}
\def\cK{{\mathscr{K}}}
\def\Layer{{\rm Layer}}
\def\cM{{\mathscr{M}}}
\def\cN{{\mathscr{N}}}
\def\cO{{\mathscr{O}}}
\def\cP{{\mathscr{P}}}
\def\cQ{{\mathscr{Q}}}
\def\cR{{\mathscr{R}}}
\def\cS{{\mathscr{S}}}
\def\Up{{\rm Up}}
\def\cU{{\mathscr{U}}}
\def\cV{{\mathscr{V}}}
\def\cW{{\mathscr{W}}}
\def\cX{{\mathscr{X}}}
\def\cY{{\mathscr{Y}}}
\def\cZ{{\mathscr{Z}}}

  \def\del{\partial}
  \def\diam{{\rm diam}}
	\def\FF{{\mathcal{F}}}
	\def\QQ{{\mathcal{Q}}}
	\def\BB{{\mathcal{B}}}
	\def\XX{{\mathcal{X}}}
	\def\PP{{\mathcal{P}}}

  \def\del{\partial}
  \def\diam{{\rm diam}}
	\def\image{{\rm Image}}
	\def\domain{{\rm Domain}}
  \def\dist{{\rm dist}}

\newcommand{\GD}[1]{{  \color{red} \textbf{Guy:} #1}}
\newcommand{\SEB}[1]{{  \color{blue} \textbf{Sylvester:} #1}}
\newcommand{\VV}[1]{{  \color{green} \textbf{Vyron:} #1}}

\begin{abstract}
A quasiconformal tree is a doubling metric tree in which the diameter of each arc is bounded above by a fixed multiple of the distance between its endpoints. In this paper we show that every quasiconformal tree bi-Lipschitz embeds in some Euclidean space, with the ambient dimension and the bi-Lipschitz constant depending only on the doubling and bounded turning constants of the tree. This answers Question 1.6 in \cite{DV}.
\end{abstract}

\maketitle

\section{Introduction}
In this paper, a \textit{(metric) tree} is a compact, connected metric space with the property that each pair of distinct points forms the endpoints of a unique arc. We show that all metric trees that satisfy two simple geometric properties can be embedded in a Euclidean space with bounded distortion.

The two properties that we impose are \textit{doubling} and \textit{bounded turning}. Recall that a metric space is called doubling if each ball in the space can be covered by $N$ balls of half the radius, for some fixed constant $N$. A metric space is called bounded turning if each pair of points $x,y$ in the space are contained in a compact, connected set whose diameter is bounded by $Cd(x,y)$, for some fixed constant $C$. In the case of metric trees, this is equivalent to saying that the unique arc joining a pair of points has diameter comparable to the distance between those points.

The class of trees satisfying these two properties were studied in detail in \cite{Kinneberg, BM1, BM2, DV}, and given the following name:

\begin{definition}
A \textit{quasiconformal tree} is a metric tree that is doubling and bounded turning.
\end{definition}

We refer the reader to \cite{BM1, DV} for more discussion on the history of quasiconformal trees, and the ways in which they arise naturally in metric geometry and complex analysis.

Our main theorem says that all quasiconformal trees bi-Lipschitz embed in some Euclidean space in a quantitative fashion: 

\begin{theorem}\label{thm:main}
If $T$ is a quasiconformal tree, then $T$ admits a bi-Lipschitz embedding into some $\RR^k$. The dimension $k$ and the bi-Lipschitz constant of the embedding depend only on the doubling and bounded turning constants of $T$.
\end{theorem}

This answers a question posed explicitly in \cite[Question 1.6]{DV}. As a reminder, a bi-Lipschitz embedding (recalled precisely in Section \ref{sec:prelim}) is an embedding that preserves all distances up to a fixed constant factor. The ``bi-Lipschitz embedding problem'' (the classification of metric spaces that admit a bi-Lipschitz embedding into some $\RR^n$), is one of the most well-known problems in the field of analysis on metric spaces. (See, e.g., \cite[Open Problem 12.3]{Heinonen}.)

The problem for quasiconformal trees is especially interesting as quasiconformal trees stand between two types of metric spaces with different behaviors. On the one hand, quasiarcs (i.e. doubling and bounded turning arcs) are a special case of quasiconformal trees and are known to always bi-Lipschitz embed into Euclidean spaces \cite[Proposition 8.1]{DV}. On the other hand, if we generalize further to what we might call ``metric graphs'' -- compact, path-connected metric spaces of topological dimension 1 -- then there exist examples that are doubling and bounded turning, generalizing quasiconformal trees, but that do not admit such embeddings. For examples, see \cite[Theorem 4.1]{Laakso} and \cite[Theorem 2.3]{LP}.

Bi-Lipschitz embeddability of quasiconformal trees was previously known in only two special cases. First, Gupta, Krauthgamer, and Lee \cite{GKL} (see also \cite{GT}) proved that if a doubling tree is \emph{geodesic} (that is, the distance between any two points of the tree is equal to the length of the unique arc that joins them), then the tree bi-Lipschitz embeds into some Euclidean space $\bR^{N}$ with $N$ depending only on the doubling constant of the tree; see also \cite{LNP} for a different proof. Note that the geodesic property is much stronger than the bounded turning property: in geodesic trees, every branch is necessarily isometric to a line segment, while in quasiconformal trees the branches may be fractal curves (e.g., the von Koch snowflake). Second, in \cite{DV} the first and third named authors proved that a quasiconformal tree bi-Lipschitz embeds in some Euclidean space if and only if the set of \emph{leaves} of the tree bi-Lipschitz embeds in some Euclidean space; see Section \ref{sec:prelim} for definitions. While the latter result is useful in cases where the embedabbility of the set of leaves is clear (e.g. the set of leaves is uniformly disconnected), there are examples of quasiconformal trees whose leaves are dense in the tree, and the result is inconclusive.

Our new construction in Theorem \ref{thm:main} owes some ideas to the well-known theorem of Assouad \cite{Assouad} that every doubling metric space admits a ``snowflake'' embedding into some Euclidean space, but new ideas are needed to improve snowflake to bi-Lipschitz in the case of quasiconformal trees. 

The crucial new idea is a novel partition of the tree into infinitely many simpler parts. Our partition bears a resemblance to the ``path partition'' used by Matou\v{s}ek in \cite[p. 231]{Matousek} to produce bi-Lipschitz embeddings of (discrete) \textit{geodesic} trees into $\ell_p$-spaces, but there are also significant differences. Other decompositions of quasiconformal trees have been produced in \cite{BM1,DV}, but they are less similar to our construction. 

Before ending the introduction, we note that, while the doubling property is necessary for the bi-Lipschitz embedabbility of any metric space into a Euclidean space, the bounded turning property is not. It is, however, a natural condition to consider and it cannot be removed from the statement of Theorem \ref{thm:main} (even for arcs), as the following example shows.

\begin{example}
Let $C\subseteq \RR$ be the standard Cantor set and $F\subseteq \RR$ be a Cantor set of positive $1$-dimensional Lebesgue measure. Let $E = F\times F\times F\subseteq \RR^3$, a Cantor set of positive $3$-dimensional Lebesgue measure. By \cite[Corollary 2]{Mcmillan}, $E$ is a \textit{tame} Cantor set, and so there is a homeomorphism from $\RR^3$ to itself that sends $C\times \{0\} \times \{0\}$ onto $E$. In particular, there is a topological arc $\Gamma\subseteq \RR^3$ containing $E$.

Now equip $\RR^3$ with the Heisenberg group metric $d_\mathbb{H}$. (See, e.g., \cite[p. 76]{Heinonen} for an introduction to the Heisenberg group.) It is a consequence of the celebrated differentiation theorem of Pansu \cite{Pansu} that the Heisenberg group (and any positive-measure subset), although doubling, admits no bi-Lipschitz embedding into any Euclidean space \cite[p. 99]{Heinonen}. Thus, $\Gamma\subseteq (\RR^3, d_{\mathbb{H}})$ is a doubling arc with no bi-Lipschitz embedding into any Euclidean space. 
\end{example}

\subsection{Outline of the paper}
Section \ref{sec:prelim} contains the basic definitions and notation used in the paper. In Section \ref{sec:decomposition}, we describe our novel way of decomposing a quasiconformal tree into useful pieces. Finally, Section \ref{sec:embedding} defines the embedding and proves that it is bi-Lipschitz.

\section{Preliminaries}\label{sec:prelim}
\subsection{Metric spaces and mappings}
We generally denote metrics on metric spaces by $d$. The \textit{diameter} of a subset $E$ in a metric space $X$ is
$$ \diam(E) = \sup\{d(x,y):x,y\in E\}.$$

Given $\epsilon>0$, an \textit{$\epsilon$-net} in a metric space $X$ is a subset $N\subseteq X$ such that $d(x,y)\geq \epsilon$ for all $x,y\in N$ and $\dist(p, N) < \epsilon$ for all $p\in X$. Each metric space has $\epsilon$-nets for every $\epsilon>0$. Moreover, each $\epsilon$-separated set in $X$ (that is, a set satisfying only the first condition to be a net) can be extended to an $\epsilon$-net of $X$. 

A mapping $f\colon X\rightarrow Y$ between two metric spaces is Lipschitz (or $L$-Lipschitz to emphasize the constant) if there is a constant $L$ such that
$$ d(f(x), f(y)) \leq L \cdot d(x,y) \text{ for all } x,y\in X.$$
The mapping $f$ is bi-Lipschitz (or $(a,b)$-bi-Lipschitz to emphasize the constants) if there are constants $a,b>0$ such that
$$ ad(x,y) \leq d(f(x),f(y)) \leq bd(x,y) \text{ for all } x,y \in X.$$

\subsection{Trees}
If $T$ is a metric tree, we write $\mathcal{L}(T)$ for the set of leaves of $T$. A point $x$ is called a leaf of $T$ if $T\setminus\{x\}$ is connected.

If $x$ and $y$ are points in a metric tree $T$, we write $[x,y]$ for the unique topological arc joining $x$ and $y$ in $T$, with the understanding that $[x,y]=\{x\}=\{y\}$ if $x=y$. Often it is convenient to orient an arc $[x,y]$ ($x\neq y$), so we assume that $[x,y]$ is equipped with a a continuous parametrization $\gamma\colon [0,1] \rightarrow [x,y]$ such that $\gamma(0)=x$ and $\gamma(1)=y$. This allows us to order subsets of $[x,y]$; we call this the ``natural order'' along the arc. If $K$ is a non-empty compact subset of $[x,y]$, we often speak of the ``first'' or ``last'' point of $K$ with respect to this ordering.

\section{Decomposing the tree}\label{sec:decomposition}

In this section, we give a new way to decompose a quasiconformal tree into useful pieces. This differs from other decompositions given in, e.g., \cite{BM1, DV}.

For the remainder of this section, we fix a tree $T$ that is doubling with constant $D$ and bounded turning with constant $1$. (One can always reduce to this case; see Section \ref{sec:embedding}.) Because $T$ is assumed to be $1$-bounded turning, we have $\diam(T)=\diam(\mathcal{L}(T))$, and we rescale so that $\diam(T)=\diam(\mathcal{L}(T))=1$. Finally, we will assume that $\mathcal{L}(T)$ -- the set of leaves of $T$ -- is finite. This assumption is not strictly necessary, and none of the constants below will depend on the cardinality of $\mathcal{L}(T)$, but it avoids certain technical difficulties. In the course of proving Theorem \ref{thm:main}, we will reduce to this case regardless. (See Section \ref{sec:embedding}.) Remark \ref{rmk:infinitecase} explains some of the annoyances that finiteness prevents.

We fix a collection $\{N_n\}_{n\geq 1}$ of $2^{-n}$-nets in $\mathcal{L}(T)$ with the property that $N_n \subseteq N_m$ if $n\leq m$. The doubling property of $T$ implies that each $N_n$ is finite. The assumption that $\mathcal{L}(T)$ is finite also implies that $N_n=\mathcal{L}(T)$ for all $n$ sufficiently large.  

We then define $T_n\subseteq T$ to be the ``convex hull'' of $N_n$, i.e., 
$$ T_n = \bigcup_{a,b\in N_n} [a,b].$$
Of course, $T_n \subseteq T_{n+1}$ for each $n$.

For each $n\geq 2$, the compact set $\overline{T_{n} \setminus T_{n-1}}$ is the disjoint union of finitely many compact connected components, which we denote $\{K^j_{n}\}_{j\in J_{n}}$. For convenience, we also set $K^1_1=T_1$ and $J_1=\{1\}$. Under our finiteness assumption, it is the case that $J_n = \emptyset$ for all $n$ sufficiently large, but none of our bounds will depend on this threshold. 

In the next two lemmas we collect some basic properties of the sets $K^j_{n}$. 

\begin{lemma}\label{lem:Kproperties}
The following properties of the sets $K^j_n$ hold:
\begin{enumerate}[(i)]
\item\label{eq:Ki} For each $n\in\mathbb{N}$ and $j\in J_n$, the set $K_{n}^j$ is a quasiconformal tree.
\item\label{eq:Kii}  If $n\geq 2$, then $K^j_n$ intersects $T_{n-1}$ in exactly one point. Moreover, for each $n \in \mathbb{N}$, $T_n \cap \overline{T_{n+1}\setminus T_n}$ contains a finite number of points.
\item If $n\in\mathbb{N}$ and $j,j' \in J_n$ with $j\neq j'$, then $K_n^{j} \cap K_n^{j'} = \emptyset$. If $n,m \in \mathbb{N}$ with $n\neq m$, $j\in J_n$, and $j'\in J_m$, then $K^i_n$ and $K^j_m$ intersect in at most one point.
\item\label{eq:Kiii} For each $n\in \mathbb{N}$ and $j\in J_n$, the set $\mathcal{L}(K^j_n)$ is contained in $T_{n-1} \cup (N_n\setminus N_{n-1})$ and contains at least one element of $N_n\setminus N_{n-1}$.  Moreover, the set $K_n^j \cap \mathcal{L}(T)$ is disjoint from $K_m^i$ for every $m\in\mathbb{N}, i\in J_m$.
\item\label{eq:Kv} We have
$$ T \subseteq \bigcup_{n\geq 1} \bigcup_{j\in J_n} K^j_n,$$
i.e., every point of $T$ is contained in some $K^j_n$. 
\end{enumerate}
\end{lemma}

\begin{proof}
For (i), we note that each $K_{n}^j$ is a compact connected subset of a metric tree, hence it is a metric tree itself. Moreover, each $K_{n}^j$ is 1-bounded turning as a subset of $T$, and $D$-doubling as a subset of a $D$-doubling space. Therefore, each $K_{n}^j$ is a quasiconformal tree.

For (ii), we start by noting that each component of $T_n \setminus T_{n-1}$ contains at least one point in $N_{n}\setminus N_{n-1}$. To see this, note that each $x\in T_n \setminus T_{n-1}$ is contained on an arc $[a,b]$ between two leaves, with $b\in N_n \setminus N_{n-1}$. If $a\in N_{n-1}$, then since $T_n$ is a tree, we must have $[x,b] \subset T_n\setminus T_{n-1}$. If $a \in N_{n}\setminus N_{n-1}$ then both arcs $[a,x]$ and $[x,b]$ can not intersect $T_{n-1}$. Either way the component containing $x$ will contain either $a$ or $b$ and that point lies in $N_n\setminus N_{n-1}$. 
 
Conversely, for any point $x\in N_{n}\setminus N_{n-1}$ there exists unique component of $T_{n} \setminus T_{n-1}$ containing $x$. Therefore, we can enumerate the components of $T_{n}\setminus T_{n-1}$ as $X_1,\dots,X_m$ for some $m\leq \card(N_{n}\setminus N_{n-1})$. Moreover, for every $i\in\{1,\dots,m\}$ we have that $\overline{X_i}$ is a metric tree  and $\overline{X_i}\cap T_{n-1}$ contains exactly one point. Now, for each $i\in J_n$ we have that $K_{n}^j = \overline{X_{i_1}}\cup\cdots \cup \overline{X_{i_j}}$ for some $i_1,\dots,i_j \in \{1,\dots,m\}$ with $\overline{X_{i_1}},\dots,\overline{X_{i_j}}$ all intersecting $T_{n-1}$ (and each other) at one single point.

Therefore, we showed that $\card(J_n) \leq \card(N_{n}\setminus N_{n-1})$ and that for any $j\in J_n$, $K^j_{n}$ intersects $T_{n-1}$ in exactly one point. It follows that $T_{n-1} \cap \overline{T_{n}\setminus T_{n-1}}$ contains at most $\card(N_{n}\setminus N_{n-1})$ many points.

For (iii), if $j,j' \in J_n$ with $j\neq j'$, then $K_n^{j} \cap K_n^{j'} = \emptyset$ since each $K_{n}^i$ is a distinct component of $\overline{T_{n+1} \setminus T_n}$. Fix now positive integers $n< m$ and $j\in J_n$ and $j'\in J_m$. Then 
\[ K_{n}^j \cap K_{m}^{j'} \subset K_{m}^{j'} \cap T_n \subset K_{m}^{j'} \cap T_{m-1}\]
where the latter intersection contains exactly one point by (ii).

For (iv), observe that 
\[ \mathcal{L}(K_n^j) \subset (K_n^j \cap T_{n-1}) \cup (\mathcal{L}(T_{n})\setminus T_{n-1}) = (K_n^j \cap T_{n-1}) \cup (N_n\setminus N_{n-1}).\]

By the proof of (ii), we know that $K_n^j$ contains a point of $N_n \setminus N_{n-1}$. Thus, $\mathcal{L}(K_n^j)$ also contains a point of $N_n \setminus N_{n-1}$.

For the ``moreover'' statement in (iv), fix a point $x \in  K_n^j \cap \mathcal{L}(T) \subseteq N_n \setminus N_{n-1}$. By (iii) we have that $x \not\in K_n^{j'}$ for any $j'\in J_n \setminus \{j\}$. If $m\neq n$ and $j'\in J_m$, then $K_m^{j'}\cap \mathcal{L}(T) \subseteq N_m \setminus N_{m-1}$ and therefore cannot contain $x$. This completes the proof of (iv).

For the last claim, we know that $N_n =\mathcal{L}(T)$ for all $n$ sufficiently large, and therefore $T_n = T$ for all $n$ sufficiently large. Given $x\in T$, choose $n\in\mathbb{N}$ to be the first index such that $x\in T_n \setminus T_{n-1}$ (viewing $T_0=\emptyset$). Then $x$ must be in some component of $\overline{T_n \setminus T_{n-1}}$, i.e., some $K_n^j$. 
\end{proof}

\begin{remark} \label{rmk:infinitecase} The simplification gained by assuming that $\mathcal{L}(T)$ is finite is most prominent in the proof of (v) in the previous lemma, which then makes Lemma \ref{lem:monotone} somewhat easier to state and prove.

If $\mathcal{L}(T)$ were infinite, then part (v) above would have to be modified to $T \setminus \mathcal{L}(T) \subset \bigcup_{n\geq 1} \bigcup_{j\in J_n} K_n^j$. To see this, consider $x\in T \setminus \mathcal{L}(T)$. Since $x\not\in \mathcal{L}(T)$, then $T\setminus \{x\}$ contains at least two components. Let $A$ be a component of $T\setminus\{x\}$ containing a point $a\in N_1$, and $B$ another component of $T\setminus\{x\}$.  The set $B$ contains at least one leaf of $T$ and is open in $T$; therefore, it contains a point $b\in N_n$ for some $n$ large. The arc $[a,b]$ contains $x$ and is contained in $T_n$. Therefore $x\in T_n$ for some $n$, and so $x\in K_m^j$ for some $m,j$. \end{remark}

\begin{lemma}\label{lem:components2}
Each tree $K^j_{n}$ has diameter at most $2^{2-n}$. Moreover, for each $n\in \mathbb{N}$ and $j\in J_n$, $K^j_n$ is a union of at most $C$ arcs, where $C$ depends only on $D$.
\end{lemma} 

\begin{proof}
For $n=1$, the first claim of the lemma is clear. For the second claim we have $\card(\mathcal{L}(K_1^1)) = \card(N_1)$, so $K^1_1$ is a union of at most $\card(N_1)$ arcs. Note also that $\card(N_1) \leq C$ for some $C$ depending only on $D$. 

Assume for the rest that $n\geq 2$. Fix $j\in J_n$ and let $x$ be the unique point in $T_{n-1} \cap K_n^j$.

For the first claim, let $[a,b]\subseteq K^j_{n}$ with $a\neq b$. The arc $[a,b]$ extends to an arc $[a',b']$ with $a',b'\in \mathcal{L}(K^j_{n})$. We consider two cases.

\emph{Case 1.} Suppose that one of the points $a',b'$, say $b'$, is the point $x$; then $a'\in N_n \setminus N_{n-1}$ by Lemma \ref{lem:Kproperties}. Choose $z\in N_{n-1}\subseteq T_{n-1}$ such that $d(z,a')< 2^{1-n}$. Then,
\[ \diam{[a,b]} \leq \diam{[a',x]} \leq \diam{[a',z]} = d(a',z) < 2^{1-n}.\]

\emph{Case 2.} Suppose that $a',b' \in N_{n}$. By Case 1 and the triangle inequality,
\[ \diam{[a,b]} \leq \diam{[a',b']} = d(a',b') \leq d(a',x) + d(b',x) < 2^{2-n}.\]
It follows that the diameter of any arc in $K_n^j$ is at most $2^{2-n}$.

We now show the second claim. By the first claim, we have that $(N_n\setminus N_{n-1})\cap K_n^j \subseteq \overline{B}(x, 2^{2-n})$. Since $N_n \setminus N_{n-1}$ is a $2^{-n}$-separated set, there exists $C>1$ depending only on $D$ such that
\[ \card{ \mathcal{L}(K_{n}^j)} \leq 1+ \card{((N_n\setminus N_{n-1})\cap K_n^j)} \leq C+1.\]
Therefore, $K^{j}_n$ is the union of at most $C$ arcs. 
\end{proof}

Set $r^1_1\in N_1 \subseteq K^1_1$ to be an arbitrary root for $K^1_1=T_1$, and let 
$$ r^j_n = \text{ the unique element of } K_n^j \cap T_{n-1},$$
(as provided by Lemma \ref{lem:Kproperties}) for each $n\geq 1$ and $j\in J_n$.

We now form a graph that encodes how close the sets $K^n_j$ are to each other. Given a constant $S\geq 1$, consider the graph $G_S=(V,E_S)$ with vertex set
$$ V = \{(n,i) : n\geq 1, i\in J_n\}$$
and edge set
$$ E_S = \{ \{(n,i), (m,j)\} : n\geq m \text{ and } \dist(K_n^i, K_m^j) \leq S2^{-\max\{n,m\}}\},$$
If $v,w\in V$, we use the notation $v\sim w$ to indicate that there is an edge between $v$ to $w$ in $E_S$, with $S$ understood from context.

We remark that the valence of the graph $G_S$ (the maximum of the valences of its vertices) may be arbitrarily large. However, we show in the next lemma that it is possible to color the vertices of $G_S$ with a fixed number of colors so that adjacent vertices have different colors.

\begin{lemma}\label{lem:coloring}
For each $S\geq 1$, there is a positive integer $A=A(S,D)$ and a ``coloring'' $\chi \colon V \rightarrow \{1, \dots, A\}$ such that if $v,w\in V$ and $v\sim w$, then $\chi(v) \neq \chi(w)$.
\end{lemma}
\begin{proof}
Given $S$, let $A=2A'+1$, where $A'=A'(S,D)$ is the maximum number of $r$-separated points in any ball of radius $200Sr$ in $T$. This is finite and depends only on $S$ and the doubling constant $D$. We now construct $\chi$ inductively. Order $V$ so that $(n,i)<(m,j)$ if and only if $n<m$ or $n=m$ and $i<j$. Set $\chi((1,1))=1$.

Suppose we have defined $\chi$ for all $v<v_0 = (n_0, i_0)$. We will show that the collection of all $v=(m,j)\in V$ such that $v<v_0$ and $v_0 \sim v$ has strictly fewer than $A$ elements. In that case, we may color $v_0$ by a color which does not match any of these, and the proof is complete.

Consider first the collection of all $v=(m,j)\in V$ such that $v<v_0$, $v_0 \sim v$ and $\diam(K_m^j)\leq 10S2^{-n}$. By Lemma \ref{lem:Kproperties}, each such $K_m^j$ contains a distinct point $n_m^j\in N_m\subseteq N_n$. Each such $n_m^j$ satisfies
$$ d(n_m^j, r_{n_0}^{i_0}) \leq \diam(K_{n_0}^{i_0}) + S2^{-n}+ \diam(K_m^j) \leq (1+11S)2^{-n}.$$
Since these points are $2^{-n}$-separated, there can be at most $A'<\frac{A}{2}$ such points, and hence less than $A/2$ such $K_m^j$.

Next we bound the total number of $v=(m,j)\in V$ such that $v<v_0$, $v_0 \sim v$ and $\diam(K_m^j)> 10S2^{-n}$. Consider such a $K_m^j$. Each such $K_m^j$ contains a point $q_m^j$ such that $\dist(q_m^j, K_{n_0}^{i_0})< 10S2^{-n}$ and $q_m^j$ is at least distance $\frac{1}{10} 2^{-n}$ from $r_m^j$. 

If $(m,j)\neq (m',j')$ for two such vertices in our graph, then $d(q_m^j, q_{m'}^{j'}) \geq \frac{1}{10}2^{-n}$: any path from $q_m^j$ to $q_{m'}^{j'}$ must pass through either $r_m^j$ or $r_{m'}^{j'}$, and so the bounded turning condition and the second defining property of $q_m^j$ yield this bound.

Thus, the collection of all such points $q_m^j$ forms a $\frac{1}{10}2^{-n}$-separated set inside $B(r_{n_0}^{i_0}, (1+10S)2^{-n})$. It follows that the collection of such $K_m^j$ has again at most $A'<A/2$ elements.

Thus, the total number of $v=(m,j)\in V$ such that $v<v_0$ and $v_0 \sim v$ is strictly less than $A$. This completes the proof.
\end{proof}

A simple consequence of the previous lemma is the following bound on how many sets $K_m^j$, with $m$ small, can intersect an arc.

\begin{lemma}\label{lem:arcintersection}
There is a constant $M=M(D)$ such that if $d(x,y)\leq 2^{-n}$, then $[x,y]$ can intersect at most $M$ distinct sets $K^m_j$ with $m\leq n$. 
\end{lemma}
\begin{proof}
Let $S=1$ and apply Lemma \ref{lem:coloring}.

Notice that if $K_{\ell}^j$ and $K_m^i$ intersect $[x,y]$ and have $\ell, m\leq n$, then  
$$\dist(K_{\ell}^j, K_m^i) \leq \diam([x,y]) \leq 2^{-n} \leq S2^{-\max\{\ell,m\}},$$
and so there is an edge $(\ell,j)\sim (m,i)$ in the graph $G_S$.

Thus, if $K_{\ell}^j$ and $K_m^i$ intersect $[x,y]$ and have $\ell, m\leq n$, then they are adjacent in $G_S$. Lemma \ref{lem:coloring} therefore says that the number of such sets is bounded by a constant depending only on $D$, which completes the proof. 
\end{proof}

Next, we examine more closely the way in which an arc $\gamma$ in $T$ can be covered by the sets $K_m^j$. Let us say that $\gamma$ \textit{traverses} $K_m^j$ if $\gamma\cap K_m^j$ contains more than one point (in which case $\gamma\cap K_m^j$ is a sub-arc of $\gamma$).

Each point $x\in \gamma$ must be contained in a set $K_m^j$ that $\gamma$ traverses. Indeed, let $\{x_i\}$ be a sequence of points in $\gamma$, all distinct from $x$, that converges to $x$. By our finiteness assumption, there are only finitely many sets $K_m^j$, and by Lemma \ref{lem:Kproperties}(v), each $x_i$ is in one of them. Therefore, a subsequence of $\{x_i\}$ is contained in a fixed $K_m^j$, which must also contain $x$ by compactness and be traversed by $\gamma$.  

In addition, note that if $\gamma$ traverses both $K_m^j$ and $K_n^i$, then $\gamma$ must intersect one of them ``first'' in the order of parametrization of the arc.

\begin{lemma}\label{lem:monotone}
Let $\gamma$ be an arc in $T$. Let $\{K_{m(i)}^{j(i)}\}_{i\in I}$ be the collection of sets $K_m^j$ that are traversed by $\gamma$, where the index set $I$ is a finite set $\{1, \dots, n\}$.  Order the sets in the order along which they intersect $\gamma$.

Then there is an index $i_0\in I$ such that $m(i) > m(i+1)$ for all $i<i_0$ and $m(i) < m(i+1)$ for all $i\geq i_0$. 
\end{lemma}
\begin{proof}
Choose $i_0\in I$ such that $m(i_0)\leq m(i)$ for all $i\in I$. 

We show by induction that for any $i< i_0$ we have $m(i+1) < m(i)$; the case $i> i_0$ is similar. First, we have that $m(i_0-1) \geq m(i_0)$. By (iii) of Lemma \ref{lem:Kproperties}, we have that $m(i_0-1) \neq m(i_0)$, so  $m(i_0-1) > m(i_0)$. Suppose now that
\[ m(i_0-l) < m(i_0-l+1) < \cdots < m(i_0).\]
Let $p$ be the unique point in $\gamma \cap K_{m(i_0-l)}^{j(i_0-l)} \cap K_{m(i_0-l+1)}^{j(i_0-l+1)}$ and let $q$ be the unique point in $\gamma \cap K_{m(i_0-l-1)}^{j(i_0-l-1)} \cap K_{m(i_0-l)}^{j(i_0-l)}$. Since $\gamma$ traverses $K_{m(i_0-l)}^{j(i_0-l)}$, we have that $p\neq q$. Moreover, by (ii) of Lemma \ref{lem:Kproperties} we have that $K_{m(i_0-l)}^{j(i_0-l)} \cap T_{m(i_0-l+1)} = \{p\}$. Therefore, $m(i_0-l-1) \geq m(i_0-l)$ and by (iii) of Lemma \ref{lem:Kproperties}, we have that $m(i_0-l-1) > m(i_0-l)$, and so the inductive step is complete.
\end{proof}

\section{Construction of the embedding and proof of Theorem \ref{thm:main}}\label{sec:embedding} 

Let $T$ be a quasiconformal tree. In this section, we prove Theorem \ref{thm:main} by showing that $T$ admits a bi-Lipschitz embedding into some Euclidean space, with distortion and dimension depending only on the doubling and bounded turning constants of $T$.

As a first simplification, we may modify $T$ by a bi-Lipschitz deformation (whose distortion depends only on the bounded turning constant) so that it is bounded turning with constant $1$ (See \cite[Lemma 2.5]{BM1}.)  Because $T$ is $1$-bounded turning, we have $\diam(T)=\diam(\mathcal{L}(T))$. We may also rescale so that $\diam(T)=\diam(\mathcal{L}(T))=1$. Let $D$ be the doubling constant of $T$, after these modifications. Throughout this section, in all statements and proofs we make these assumptions.

Finally, we assume without loss of generality that the set $\mathcal{L}(T)$ of leaves of $T$ is finite. This assumption is justified by the following basic fact in metric embeddings: A compact metric space $X$ admits an $(a,b)$-bi-Lipschitz embedding into $\RR^k$ if and only if every finite subset of $X$ admits an $(a,b)$-bi-Lipschitz embedding into $\RR^k$. (See, e.g., \cite[equation (1)]{NN} or \cite[Lemma 4.9]{Assouad}.) Thus, if we prove Theorem \ref{thm:main} only for trees with finitely many leaves and we wish to embed an arbitrary quasiconformal tree $T$, we may apply the theorem to the convex hull of every finite subset of $T$ (which are all uniformly doubling, bounded turning trees), and conclude that Theorem \ref{thm:main} holds for $T$ by the ``basic fact'' mentioned above.

To prove Theorem \ref{thm:main}, it thus suffices to embed a $D$-doubling, $1$-bounded-turning tree $T$ with finite leaf set into some $\RR^k$, with $k$ and the distortion depending only on $D$. These assumptions on $T$ are now in force. We first apply the construction of the previous section to obtain sets $K_n^j$ satisfying the properties of Section \ref{sec:decomposition}.

Our first step is to construct ``local'' bi-Lipschitz embeddings on each piece $K_n^j$.
\begin{lemma}\label{lem:localembedding}
There are constants $L=L(D)\geq 1$ and $d=d(D)\in\mathbb{N}$ such that, for each $n\geq 1$ and $j\in J_n$, there is a $(1,L)$-bi-Lipschitz embedding
$$f_n^j\colon K_n^j \rightarrow \RR^d$$
such that $f_n^j(r_n^j)=0$.

Moreover, this embedding extends to an L-Lipschitz map $f_n^j\colon T \rightarrow \RR^d$ that is constant on each component of $T\setminus K_{n}^j$. In particular, $f_n^j$ is constant on $K_m^i$ if $(m,i)\neq (n,j)$.
\end{lemma}

\begin{proof}
Fix any $K_n^j$. The set $K_n^j$ is a union of at most $C=C(D)$ arcs, by Lemma \ref{lem:components2}. Each such arc is $D$-doubling and $1$-bounded turning, and therefore admits a bi-Lipschitz embedding into $\RR^{d_0}$, $d_0=d_0(D)$ and bi-Lipschitz constants depending only on $D$. (This is \cite[Proposition 8.1]{DV}.) By \cite[Theorem 8.2]{LP}, we see that $K_n^j$ then admits an bi-Lipschitz embedding $f_n^j$ into some $\RR^{d}$, for $d=d(D)$ and constants dpending on $D$ and $C=C(D)$. By rescaling and translating this embedding, we can ensure that it is $(1,L)$ bi-Lipschitz for $L=L(D)$ and maps $r_n^j$ to the origin.

To construct the Lipschitz extension, let $K$ be a component of $T\setminus K_n^j$. Then $\overline{K}\cap K_n^j$ consists of exactly one point, $p$. Indeed, this intersection contains at least one point as $T$ is connected. Suppose it contained distinct points $p\neq q$, in which case it would also contain the arc $[p,q]$. Let $p',q'$ be a points of $K$ with distance at most $d(p,q)/10$ from $p,q$, respectively. Then $[p',q'] \cup [q',q] \cup [q,p]$ would contain an arc joining $p$ and $p'$ of diameter at least $d(p,q)/2 > d(p,p')$, contradicting the bounded turning property.

We thus set $f_n^j(x) = f_n^j(p)$ for all $x\in K$. This extends $f_n^j$ to all of $T$ in a way that it constant on each component of $T\setminus K_n^j$.

To prove that the extension is $L$-Lipschitz, consider $x, y\in T$. If $[x,y]$ is disjoint from $K_n^j$, then $x$ and $y$ lie in the same component of $T\setminus K_n^j$ and thus $|f(x)-f(y)|=0$. Otherwise, let $p$ and $q$ be the first and last points, respectively, on $[x,y]\cap K_n^j$. Then either $p=x$ or $p$ is the unique point in both $K_n^j$ and in the component of $T\setminus K_n^j$ containing $x$; in either case, we have $f_n^j(x)=f_n^j(p)$. Similarly, $f_n^j(y)=f_n^j(q)$. Therefore,
$$ |f_n^j(x)-f_n^j(y)|=|f_n^j(p)-f_n^j(q)| \leq L d(p,q) \leq L\diam([p,q]) \leq L\diam([x,y]) \leq Ld(x,y),$$
using the known fact that $f_n^j$ is $L$-Lipschitz on $K_n^j$ itself.

For the final, ``in particular'' statement, suppose $(n,j)\neq (m,i)$. By Lemma \ref{lem:Kproperties}, $K_n^j \cap K_m^i$ contains at most one point, $p$. If $x,y\in K_m^i$ and $[x,y]$ does not contain this point (or there is no such point), then $[x,y]$ lies in a component of $T\setminus K_n^j$ and so $f_n^j(x)=f_n^j(y)$. Otherwise, $x$ and $p$ lie in the closure of a component of $T\setminus K_n^j$, and the same holds for $p$ and $y$, and so $f_n^j(x)=f_n^j(p)=f_n^j(y)$. Thus, $f_n^j$ is constant on $K_m^i$.
\end{proof}

\begin{lemma}\label{lem:localzero}
If $x\in T$, $n\in\mathbb{N}$, $j\in J_n$, then $f_n^j(x)=0$ unless the arc $[r_1^1, x]$ traverses $K_n^j$.
\end{lemma}
\begin{proof}
First, observe that $f_n^j(r_1^1)=0$ for every $n\in\mathbb{N}$ and $j\in J_n$. If $n=1$ this is by definition of $f_1^1$. Otherwise, this is because $f_n^j(r_n^j)=0$ by definition of $f_n^j$ and $r_n^j$ and $r_1^1$ are joined by an arc in $T_{n-1}$ whose only intersection with $K_n^j$ is $r_n^j$. Therefore, $r_n^j$ and $r_1^1$ lie in the closure of a common component of $T\setminus K_n^j$.

Now suppose that $[r_1^1, x]$ does not traverse $K_n^j$. 

If $[r_1^1,x]$ is disjoint from $K_n^j$, then $r_1^1$ and $x$ are in the same component of $T\setminus K_n^j$ and $f_n^j(x)=f_n^j(r_1^1)=0$.

If $[r_1^1,x]$ is not disjoint from $K_n^j$, then it intersects $K_n^j$ in exactly one point, $y$. It follows that $f_n^j(r_1^1)=f_n^j(y)$ and $f_n^j(y)=f_n^j(x)$, hence $f_n^j(x)=0$.
\end{proof}

We fix constants before proceeding. Let $M=M(D)$ be as in Lemma \ref{lem:arcintersection}, and let $L=L(D), d=d(D)$ be as in Lemma \ref{lem:localembedding}. Let $N\in\mathbb{N}$ be chosen sufficiently large depending only on $L$ and $M$ (hence only on $D$); specifically, we will require that
$$ \frac{1}{2\sqrt{2N+M+1}} - L2^{3-N} > 0.$$
Let $S=2^{N}$. Finally, let $A=A(D)$ be the number of colors produced by Lemma \ref{lem:coloring} using this choice of $S$.

We now define our bi-Lipschitz embedding $f\colon T \rightarrow \RR^{Ad}$. We use the $L$-Lipschitz mappings $f_n^j\colon T \rightarrow \RR^d$ defined in Lemma \ref{lem:localembedding}. View $\RR^{Ad}$ as $\bigoplus_{\ell=1}^A \RR^d_\ell$, with each $\RR^d_\ell$ a copy of $\RR^d$.  Making a simple adjustment to $f_n^j$, we post-compose with an isometric embedding so that 
$$ f_n^j \colon T \rightarrow \RR^d_{\chi((n,j))} \subseteq \RR^{Ad},$$
where $\chi$ is the coloring from Lemma \ref{lem:coloring}, and $f_n^j(r_n^j) = 0$.

We now define $f\colon T \rightarrow \RR^{Ad}$ by
\begin{equation}\label{eq:fdef}
f(x) = \sum_{n\in \mathbb{N}} \sum_{j\in J_n} f_n^j(x).
\end{equation}
Because of our assumption that the leaf set of $T$ is finite, only finitely many of the sets $J_n$ are non-empty, and hence this sum is always finite.

\begin{proof}[Proof of Theorem \ref{thm:main}] 

Recall that, by the discussion at the beginning of this section, we have reduced the embedding problem to the case of a $1$-bounded turning tree $T$ with $\mathcal{L}(T)$ finite and $\diam(T)=\diam(\mathcal{L}(T))=1$. We will show that, under these assumptions, $f$ as defined in \eqref{eq:fdef} is $(a,b)$-bi-Lipschitz with constants $a$ and $b$, depending only on $D$, to be described below.

Fix $x,y \in T$ with $x\neq y$. 

Let $\{K_{m(i)}^{j(i)}\}_{i\in I}$ be the (finite) collection of sets $K_m^j$ that are traversed by $[x,y]$. Here we set $I$ to be a finite index set $\{1, 2, \dots, \max(I)\}$, and the sets are ordered in the natural order along the arc $[x,y]$. By Lemma \ref{lem:monotone}, there is an index $i_0\in I$ such that $m(i) > m(i+1)$ for all $i<i_0$ and $m(i) < m(i+1)$ for all $i\geq i_0$. 

For each $i\in I$, let $p_i$ be the first point on $[x,y] \cap K_{m(i)}^{j(i)}$. We also add one additional point at the end of $I$ and set $p_{\max(I)}=y$. Because $x$ is contained in a set traversed by $\gamma$ (see the remark above Lemma \ref{lem:monotone}), we also have $p_{1} = x$.

We can now write a telescoping sum:
\begin{equation}\label{eq:telescope}
 f(x) - f(y) = \sum_{i\in I} \left(f(p_i) - f(p_{i+1})\right).
\end{equation}

Let $n\in \bN$ be such that
$$ 2^{-n-1} \leq d(x,y) \leq 2^{-n}.$$

%Recall that by Lemma \ref{lem:monotone}, there is an index $i_0$ such that $m(i)>m(i+1)$ for $i<i_0$ and $m(i)<m(i+1)$ for $i\geq i_0$.

Define indices $i_*$ and $i^*$ in $I$ as follows:
$$ i_* = \min\{ i\in I: m(i)\leq n\} \text{ and } i^* = \max\{i\in I: m(i)\leq n\}.$$
(We set $i_*=i^*=i_0$ if the associated set is empty.)
By Lemmas \ref{lem:arcintersection} and \ref{lem:monotone}, we have $i^*-i_* \leq M=M(D)$. 

The restriction of $f$ to a single $K_m^j$ is $(1,L)$-bi-Lipschitz, because $f_m^j$ is $(1,L)$-bi-Lipschitz on $K_m^j$ and each $f_n^i$ for $(n,i)\neq (m,j)$ is constant on $K_m^j$. (See Lemma \ref{lem:localembedding}.) Note also that $p_i$ and $p_{i+1}$ are both always in $K_{m(i)}^{j(i)}$. Therefore, we always have
\begin{equation}\label{eq:fbound1}
 d(p_i, p_{i+1}) \leq |f(p_i) - f(p_{i+1})|\leq L\diam(K_{m(i)}^{j(i)}) \leq L2^{2-m(i)}.
\end{equation}
In addition, the bounded turning condition and the fact that $p_i\in [x,y]$ ensures that we also always have
\begin{equation}\label{eq:fbound2}
|f(p_i) - f(p_{i+1})| \leq Ld(p_i, p_{i+1}) \leq Ld(x,y).
\end{equation}

Using equations \eqref{eq:fbound1} and \eqref{eq:fbound2}, we obtain:
\begin{align*}
|f(x) - f(y)| &\leq \sum_{i\in I} |f(p_i) - f(p_{i+1})|\\
&\leq \sum_{i\in I, i<i_*}  |f(p_i) - f(p_{i+1})| + \sum_{i\in I, i_*\leq i \leq i^*}  |f(p_i) - f(p_{i+1})| + \sum_{i\in I, i>i^*}  |f(p_i) - f(p_{i+1})|\\
&\leq L\left(\sum_{i\in I, i<i_*} 2^{2-m(i)} + \sum_{i\in I, i_*\leq i \leq i^*} d(x,y)  + \sum_{i\in I, i>i^*}  2^{2-m(i)}\right)\\
&\lesssim L 2^{-m(i_*-1)} + Md(x,y) + L 2^{-m(i^*+1)}\\
&\lesssim L 2^{-n} + Md(x,y) + L2^{-n}\\
&\lesssim d(x,y) 
\end{align*}
This proves that $f$ is Lipschitz (with constant depending only on $D$).

For the lower bound, we similarly break up the sum in \eqref{eq:telescope} into three pieces, but at different points. Define indices $\underline{i}$ and $\overline{i}$ in $I$ as follows:
$$ \underline{i} = \min\{ i\in I: m(i)\leq n+N\} \text{ and } \overline{i} = \max\{i\in I: m(i)\leq n+N\}.$$
(We interpret $\underline{i}=\overline{i}=i_0$if the associated set is empty, although in fact the proof shows that this cannot happen with our choice of $N$.)
Note that $\overline{i}-\underline{i}\leq 2N + M$ by Lemmas \ref{lem:arcintersection} and \ref{lem:monotone}.

Note that 
\begin{equation}\label{eq:tails}
\sum_{i>\overline{i}} d(p_i, p_{i+1}) \leq \sum_{i>\overline{i}} \diam(K_{m(i)}^{j(i)}) \leq \sum_{i>\overline{i}} 2^{2-m(i)} \leq 2^{2-(n+N)}.
\end{equation}
The same bound holds for the sum over $i< \underline{i}$.

Now we note that if $\underline{i}\leq i<i' \leq \overline{i}$, then $f(p_i)-f(p_{i+1})$ and $f(p_{i'})-f(p_{i'+1})$ lie in orthogonal subspaces $\RR^d_\ell$ and $\RR^d_{\ell'}$. To see this, note that $p_i$ and $p_{i+1}$ are both in $K_{m(i)}^{j(i)}$, and similarly for $p_{i'}$ and $p_{i'+1}$. Therefore, recalling the last statement of Lemma \ref{lem:localembedding},
$$ f(p_i) - f(p_{i+1}) \in \RR^d_\ell,$$
where $\ell = \chi(m(i),j(i))$. Similarly, $f(p_i) - f(p_{i+1}) \in \RR^d_{\ell'}$, where $\ell' = \chi(m(i'),j(i'))$. To see that $\ell\neq \ell'$, observe that
$$ \dist(K_{m(i)}^{j(i)}, K_{m(i')}^{j(i')}) \leq d(x,y) \leq 2^{-n} \leq 2^N 2^{-\max\{m(i), m(i')\}} = S2^{-\max\{m(i), m(i')\}},$$
so $\chi((m(i),j(i)) \neq \chi((m(i'),j(i'))$ by the defining property of the coloring $\chi$ from Lemma \ref{lem:coloring}, and our choice of $S=2^N$.

Therefore, by using \eqref{eq:fbound1},  the Cauchy--Schwarz inequality, the fact that $\sum_I d(p_i,p_{i+1})\geq d(x,y)$, and  \eqref{eq:tails}, we obtain  
\begin{align*}
\left| \sum_{\underline{i}\leq i \leq \overline{i}} f(p_i)-f(p_{i+1}) \right| &= \left(\sum_{\underline{i}\leq i \leq \overline{i}} |f(p_i)-f(p_{i+1})|^2 \right)^{1/2}\\ 
&\geq \left(\sum_{\underline{i}\leq i \leq \overline{i}} d(p_i, p_{i+1})^2\right)^{1/2}\\
&\geq \frac{1}{\sqrt{\overline{i}-\underline{i}+1}} \sum_{\underline{i}\leq i \leq \overline{i}} d(p_i, p_{i+1})\\
&\geq \frac{1}{\sqrt{\overline{i}-\underline{i}+1}}\left( d(x,y) - \sum_{i<\underline{i} \text{ or } i > \overline{i}} d(p_i, p_{i+1})\right)\\
&\geq \frac{1}{\sqrt{2N+M+1}}( d(x,y)  - 2^{3-(n+N)})\\
&\geq \left(\frac{1}{\sqrt{2N+M+1}} - 2^{3-N}\right)d(x,y)\\
&\geq \frac{1}{2\sqrt{2N+M+1}} d(x,y).
\end{align*}

Again using \eqref{eq:tails} and the fact that the restriction of $f$ to each $K_m^j$ is $L$-Lipschitz, we therefore have
\begin{align*}
|f(x)-f(y)| &\geq \left|\sum_{\underline{i}\leq i \leq \overline{i}} f(p_i)-f(p_{i+1})\right|  - \sum_{i<\underline{i} \text{ or } i > \overline{i}} |f(p_i)-f(p_{i+1})|\\
&\geq \left|\sum_{\underline{i}\leq i \leq \overline{i}} f(p_i)-f(p_{i+1})\right|  - L\sum_{i<\underline{i} \text{ or } i > \overline{i}} d(p_i,p_{i+1})\\
&\geq \frac{1}{2\sqrt{2N+M+1}} d(x,y) -  L2^{3-(n+N)}\\
&\geq \left(\frac{1}{2\sqrt{2N+M+1}} - L2^{3-N}\right) d(x,y)\\
&\gtrsim d(x,y)
\end{align*}
by our choice of $N$.
\end{proof}

\bibliography{quasitrees-ref}
\bibliographystyle{amsbeta}

\end{document}